\documentclass[12pt, twoside]{article}
\usepackage{amsmath,amsthm,amssymb, graphicx}
\usepackage{times}
\usepackage{enumerate}

\def\titlerunning#1{\gdef\titrun{#1}}
\makeatletter
\def\author#1{\gdef\autrun{\def\and{\unskip, }#1}\gdef\@author{#1}}
\def\address#1{{\def\and{\\\hspace*{18pt}}\renewcommand{\thefootnote}{}%
\footnote {#1}}%
\markboth{\autrun}{\titrun}}
\makeatother
\def\email#1{e-mail: #1}
\def\subjclass#1{{\renewcommand{\thefootnote}{}%
\footnote{\emph{Mathematics Subject Classification (2010):} #1}}}
\def\keywords#1{\par\medskip
\noindent\textbf{Keywords.} #1}

\newtheorem{theorem}{Theorem}[section]

\newtheorem{lemma}[theorem]{Lemma}
\newtheorem{proposition}[theorem]{Proposition}
\theoremstyle{remark}
\newtheorem{definition}[theorem]{Definition}
\newtheorem{example}[theorem]{Example}

\newtheorem{question}[theorem]{Question}

\numberwithin{equation}{section}

\newcommand{\CC}{\mathbb{C}}
\newcommand{\DD}{\mathbb{D}}
\newcommand{\HH}{\mathbb{H}}
\newcommand{\RR}{\mathbb{R}}
\newcommand{\TT}{\mathbb{T}}
\DeclareMathOperator{\dist}{dist}
\DeclareMathOperator{\Lip}{Lip}
\DeclareMathOperator{\re}{Re}
\DeclareMathOperator{\im}{Im}
\DeclareMathOperator{\diam}{diam}
\DeclareMathOperator{\sgn}{sign}

\begin{document}


\baselineskip=17pt


\titlerunning{Optimal extension of Lipschitz embeddings in the plane}

\title{Optimal extension of Lipschitz embeddings in the plane}

\author{Leonid V. Kovalev}

\date{\today}

\maketitle

\address{L. V. Kovalev: 215 Carnegie, Mathematics Department, Syracuse University, Syracuse, NY 13244, USA; \email{lvkovale@syr.edu}}



\begin{abstract} 
We prove that every bi-Lipschitz embedding of the unit circle into the plane can be extended to a bi-Lipschitz map of the unit disk with linear bounds on the constants involved. This answers a question raised by Daneri and Pratelli. Furthermore, every Lipschitz embedding of the circle extends to a Lipschitz homeomorphism of the plane, again with a linear bound on the constant.   

\end{abstract}

\section{Introduction}

Every simple closed rectifiable curve in the plane $\RR^2$ is the image of an injective Lipschitz map $f$ defined of the unit circle $\TT$, for example via the constant speed parameterization. Two classical extension theorems apply to such $f$: the Schoenflies theorem~\cite[Theorem 10.4]{Moise} says that $f$ extends to a homeomorphism $F\colon \RR^2\to\RR^2$, while the Kirszbraun theorem~\cite[Theorem 1.34]{BB} says that $f$ extends to a Lipschitz map $F\colon \RR^2\to\RR^2$. But is there an extension with both properties?

Despite the simplicity of the question, an answer to it is not found in the literature. The bi-Lipschitz extension problem, in which both $f$ and $f^{-1}$ are assumed Lipschitz continuous, has been studied by many authors~\cite{ATV, ATV2, AS, DP, JK, KVW, Lat2, Mac, T, Tu, Tu2}. A common theme in the existing results is that the Lipschitz constant $\Lip(F)$  depends on $\Lip(f^{-1})$ as well as on $\Lip(f)$, and therefore no conclusion is reached if $f^{-1}$ is not Lipschitz. The curves that admit bi-Lipschitz parameterization by $\TT$ are precisely the chord-arc curves~\cite[Theorem 7.9]{Pomm}, and the chord-arc property is much stronger than rectifiability. For example, cardioids and astroids do not have this property.

Even when the embedding $f$ is bi-Lipschitz, there remains the question of controlling the Lipschitz constants of its extension in the optimal way, which is $\Lip(F)\le C\Lip(f)$ and $\Lip(F^{-1})\le C\Lip(f^{-1})$ with $C$ independent of $f$. The problem of extending an embedding $f\colon \TT\to\CC$ to a map of the unit disk $\DD$ in this way was posed by Daneri and Pratelli in~\cite{DP}. Their paper was followed by a number of quantitative extension theorems~\cite{Ka, Ka2, Ko12, Ko18, Radici} but the problem remained open until now.  

Our main result, Theorem~\ref{main-theorem}, solves the problem posed by Daneri and Pratelli, as well as the aforementioned problem of extending a Lipschitz (not necessarily bi-Lipschitz) embedding to a Lipschitz homeomorphism.  

\begin{theorem}\label{main-theorem}
Suppose $f\colon\TT\to\CC$ is an embedding such that 
\begin{equation}\label{BLass}
\ell|a-b|\le  |f(a)-f(b)| \le L |a-b|,\quad a,b \in\TT 
\end{equation}
for some $\ell, L\in [0, \infty]$. Then $f$ extends to a homeomorphism $F\colon\CC\to\CC$ that  satisfies 
\begin{equation}\label{main-conclusion}
10^{-14} \ell|a-b|\le  |F(a)-F(b)| \le 2\cdot 10^{14} L |a-b|,\quad a,b \in\overline{\DD},
\end{equation}
as well as
\begin{equation}\label{main-conclusion2}
10^{-25} \frac{\ell^2}{L} |a-b|\le  |F(a)-F(b)| \le 10^{28} L |a-b|,\quad a,b \in \CC. 
\end{equation}
\end{theorem}

Note that $\ell=0$ is allowed in Theorem~\ref{main-theorem}, in which case $f^{-1}$ is not required to be Lipschitz. Similarly, $L=\infty$ is allowed, in which case $f$ is not required to be Lipschitz. 

Example~\ref{keyhole} shows that the factor in the lower bound in~\eqref{main-conclusion2} cannot be of the form $c\ell$ with a constant $c>0$; more generally it cannot be a positive quantity that depends on $\ell$ alone. 

The existence of an extension with the property~\eqref{main-conclusion2} has been proved in~\cite[Theorem~1.1]{Ko18}. The main achievement of the present paper is improving the lower bound for $|F(a)-F(b)|$ with $a,b\in \DD$ to the optimal form $c\ell|a-b|$ with a universal constant $c>0$. 

The  paper is organized as follows. Section~\ref{BAextsec} concerns the extension of suitably normalized circle homeomorphisms. Section~\ref{harmestsec} collects auxiliary estimates for harmonic measure and conformal maps, following~\cite{Ko18}. Theorem~\ref{main-theorem} is proved in \S\ref{sec-main-theorem-proof}. Finally, \S\ref{sec:examples} presents examples that show the estimates in Theorem~\ref{main-theorem} are in a certain way optimal. 

\emph{Notation and terminology}: $\TT$ is the unit circle in the plane, which is identified with the complex plane $\CC$. We write $|\gamma|$ for the length of an arc $\gamma\subset \TT$. The open unit disk is denoted by $\DD$. A map $f$ is \emph{Lipschitz} if the quantity
\[
\Lip(f) := \sup_{a\ne b}\frac{|f(a)-f(b)|}{|a-b|}, 
\]
called the Lipschitz constant of $f$, is finite. Such a map is \emph{bi-Lipschitz} if its inverse is Lipschitz as well. An \emph{embedding} is a map that is a homeomorphism onto its image. The notation $a\wedge b$ means $\min(a, b)$. The diameter $\diam A$ and distance $\dist(A, B)$ are always taken with respect to the Euclidean metric on $\RR^2$.

\section{Extension of a normalized circle homeomorphism}\label{BAextsec}

When working with a circle homeomorphism, we would like to prevent it from mapping a short arc onto an arc with short complement. The following definition makes this property precise. 

\begin{definition}\label{normalized-homeo}
A circle homeomorphism $\psi\colon \TT\to\TT$ is \emph{normalized} if for every arc $\gamma\subset \TT$ of length $2\pi/3$ its image $f(\gamma)$ has length at most $4\pi/3$. 
\end{definition}

Recall that a M\"obius transformation is a map of the form $\mu(z) = \nu (z-a)/(1-\bar a z)$ with $|a|<1$ and $\nu\in\TT$. Such $\mu$ leaves $\TT$ invariant. 

\begin{lemma}\label{normalize}
For every circle homeomorphism $\psi\colon \TT\to\TT$ there exists a M\"obius transformation $\mu\colon \TT\to\TT$ such that $\psi\circ \mu$ is normalized. 
\end{lemma}

\begin{proof} Since M\"obius transformations act transitively on the triples of points on $\TT$, we can choose $\mu$ so that the composition $\phi = \psi\circ \mu$ maps the set of the cubic roots of unity $C = \{e^{2\pi i k/3} \colon k=0, 1, 2\}$ onto itself. An arc $\gamma$  of length $2\pi/3$ is disjoint from at least one of the connected components of $\TT\setminus C$. Therefore, $\phi(\gamma)$ is also disjoint from one of these components, which implies   $|\phi(\gamma)|\le 4\pi/3$. 

Alternatively, one can choose $\mu$ so that the composition $\phi=\psi\circ\mu$ has mean zero; this is done in~\cite[p.~26]{DE}. If $|\phi(\gamma)| > 4\pi/3$, then there exists  $\lambda\in \TT$ such that $\re(\lambda w)>1/2$ for all  $w\in \TT\setminus \phi(\gamma)$. This means $\re (\lambda \phi)>1/2$ on $\TT\setminus\gamma$, which leads to a contradiction:  
\[
0 = \re \int_\TT \lambda \phi > \frac12|\TT\setminus \gamma| - |\gamma| >  \frac12 \frac{4\pi}{3} - \frac{2\pi}{3} = 0.
\qedhere \]
\end{proof}

To state the main result of this section, we need to introduce four circular arcs  that will appear throughout the paper. Given a point $z = re^{i\theta}$ with $0<r<1$, let $\delta = \log (1/r)$ and introduce four arcs of the unit circle $\TT$:  
\begin{equation}\label{gammacurves}
\begin{split}
&\gamma_1 = \{e^{it}\colon \theta-2\delta \le  t\le \theta - \delta \}, \\
&\gamma_2 = \{e^{it}\colon \theta-\delta \le  t\le \theta - \delta/2 \}, \\
&\gamma_3 = \{e^{it}\colon \theta+\delta/2 \le  t\le \theta + \delta \}, \\
&\gamma_4 = \{e^{it}\colon \theta+\delta \le  t\le \theta + 2\delta\}.
\end{split}
\end{equation}
These arcs may overlap, and even coincide with $\TT$ when $r$ is small. 

\begin{proposition} \label{homeodist} Let $\psi\colon \TT\to\TT$ be a normalized circle homeomorphism. There exists a homeomorphic extension $\Psi\colon \DD\to\DD$ which is also a diffeomorphism of $\DD\setminus\{0\}$ onto itself, and such that for every point $z\in\DD\setminus\{0\}$ the derivative matrix $D\Psi(z)$ satisfies
\begin{equation}\label{BAderest1}
\| D\Psi(z) \| \le  10^7\times 
\begin{cases}
\dfrac{\dist(\psi(\gamma_1), \psi(\gamma_4))}{\log(1/|z|)},\quad 
&e^{-\pi/6} < |z| < 1 \\ 
1, \quad &0< |z| \le e^{-\pi/6}
\end{cases}
\end{equation}
\begin{equation}\label{BAderest2}
\| D\Psi(z)^{-1} \| \le 
10^7\times \begin{cases}
\dfrac{\log(1/|z|)}{\diam \psi(\gamma_1) \wedge \diam \psi(\gamma_3)},\quad &e^{-4\pi} < |z| < 1 \\ 
1 ,\quad &0< |z| \le e^{-4\pi}
\end{cases}
\end{equation}
where the arcs $\gamma_j$ are as in~\eqref{gammacurves}.
\end{proposition} 

\begin{proof} We may and do assume $\psi$ is sense-preserving. Then $\psi$ lifts to an increasing homeomorphism $\chi$ of the real line onto itself, which satisfies $\psi (e^{it}) = e^{i \chi(t)}$ for all $t\in \RR$, and 
\begin{equation}\label{step2pi}
\chi(t+2\pi)=\chi(t)+2\pi. 
\end{equation}
Let us note two elementary consequences of~\eqref{step2pi}:
\begin{equation}\label{large-distance}
2\pi \left\lfloor \frac{t-s}{2\pi} \right\rfloor \le \chi(t)-\chi(s) \le 2\pi \left\lceil \frac{t-s}{2\pi} \right\rceil,
\quad s\le t;
\end{equation}
\begin{equation}\label{notfar}
|\chi(t)-\chi(0)-t| \le 2\pi, \quad t\in\RR.
\end{equation}

Let $\chi_e$ denote the following variant of the Beurling-Ahlfors extension of $\chi$: 
\begin{equation}\label{ext1}
\chi_e(x+iy)=\frac{1}{2}\int_{-1}^{1} \chi(x+ty)(1+2i\sgn t)\,dt.
\end{equation}
This is a diffeomorphism of the upper halfplane $\HH$ onto itself~\cite{Ahb, BA}, which satisfies the following inequalities ~\cite[(4.3-5)]{Ko18}:
\begin{equation}\label{maxstr}
\| D\chi_e(x+iy) \| \le 2\,\frac{\chi(x+y)-\chi(x-y)}{y};
\end{equation}
\begin{equation}\label{minstr}
\| D\chi_e(x+iy)^{-1} \| \le \frac{4y}{(\chi(x+y)-\chi(x+\frac{y}2))\wedge  (\chi(x-\frac{y}2)-\chi(x-y))};
\end{equation}
\begin{equation}\label{imext}
|\im \chi_e(x+iy) - y| \le 4\pi.
\end{equation}
By construction, $\chi_e(z+2\pi)=\chi_e(z)+2\pi$ for all $z\in\HH$,  which allows us to define the desired map $\Psi\colon \DD\to \DD$ by
\begin{equation}\label{halfplane2disk}
\Psi(e^{iz}) = \exp(i \chi_e(z)),\quad  \Psi(0)=0.
\end{equation}
This is a diffeomorphism of the punctured disk $\DD\setminus \{0\}$ onto itself, and also a homeomorphism of $\DD$ onto $\DD$. Using~\eqref{imext} and the chain rule, we obtain
\begin{equation}\label{chainPsi}
\| D\Psi(e^{iz})\| \le e^{4\pi} \| D\chi_e(z)\|,\quad 
\| D\Psi(e^{iz})^{-1} \| \le e^{4\pi} \| D\chi_e(z)^{-1}\|.
\end{equation}
It remains to prove the estimates~\eqref{BAderest1} and~\eqref{BAderest2}.
Note that $\Psi(re^{i\theta}) = \exp(i \chi_e(\theta + i\delta))$ where $\delta = \log(1/r)$.

\textit{Proof of~\eqref{BAderest1}.} In view of~\eqref{chainPsi}, the estimate ~\eqref{maxstr} yields
\begin{equation}\label{dPsi1}
\| D\Psi(z)\|  \le 2e^{4\pi}\,  \frac{\chi(\theta+\delta)- \chi(\theta-\delta)}{\delta}.
\end{equation}
Suppose $r>e^{-\pi/6}$. Then $\gamma_1\cup \gamma_4\subset \gamma$ where $\gamma = \{e^{it}\colon |t-\theta|\le \pi/3\}$. Since $\psi$ is normalized, $\psi(\gamma)$ has length at most $4\pi/3$. Considering that the Euclidean distance between the endpoints of an arc of length $\alpha$  is $2\sin (\alpha/2)$, we find that  the ratio of arclength distance to Euclidean distance on $\psi(\gamma)$ is bounded above by
\[
\frac{\alpha}{2\sin \frac{\alpha}{2}}\bigg|_{\alpha=4\pi/3}
 = \frac{4\pi}{3\sqrt{3}}.
\]
Therefore, 
\begin{equation}\label{arcdist}
\dist(\psi(\gamma_1), \psi(\gamma_4)) \ge \frac{3\sqrt{3}}{4\pi} (\chi(\theta+\delta)- \chi(\theta-\delta)). 
\end{equation}
Inequality~\eqref{BAderest1} follows from~\eqref{dPsi1} and~\eqref{arcdist}.

Next consider the case $0<r\le e^{-\pi/6}$. By virtue of ~\eqref{step2pi},
\begin{equation}\label{14delta}
\chi(\theta+\delta)- \chi(\theta-\delta) \le 2\pi 
\left \lceil \frac{\delta}{\pi}\right\rceil \le 
2\delta + 2\pi \le 14\delta 
\end{equation}
where the last step uses $\delta\ge \pi/6$. From~\eqref{dPsi1} and~\eqref{14delta} it follows that  $\| D\Psi(z) \|  \le 28e^{4\pi}<10^7$ in this case.

\textit{Proof of~\eqref{BAderest2}.}
From ~\eqref{minstr} and ~\eqref{chainPsi} we have
\begin{equation}\label{dPsi2}
\| D\Psi(z)^{-1}\|   \le \frac{4 e^{4\pi} \delta }{(\chi(\theta+\delta)-\chi(\theta+\delta/2)) \wedge (\chi(\theta-\delta/2)-\chi(\theta-\delta))}.
\end{equation}
The denominator of ~\eqref{dPsi2} is the length of the shorter of the arcs $\psi(\gamma_2), \psi(\gamma_3)$. It is bounded from below by $\diam \psi(\gamma_2) \wedge \diam \psi(\gamma_3)$, which yields the first part of ~\eqref{BAderest2} (the restriction $r>e^{-4\pi}$ is immaterial here).

For small $r$, specifically $0<r\le e^{-4\pi}$, we require a uniform estimate. Since $\delta\ge 4\pi$, we can use the inequality
$\lfloor \delta/(4\pi)\rfloor \ge \delta/(8\pi)$ in~\eqref{large-distance}, obtaining 
\[
\chi(\theta+\delta)-\chi(\theta+\delta/2) \ge 
2\pi \left\lfloor \frac{\delta}{4\pi}\right \rfloor 
\ge 2\pi \frac{\delta}{8\pi} = \frac{\delta}{4}.
\]
The same bound holds for $\chi(\theta-\delta/2)-\chi(\theta-\delta)$. Thus, the right hand side of ~\eqref{dPsi2} is bounded by $16e^{4\pi}<10^7$. The proof of Proposition~\ref{homeodist} is complete.
\end{proof}

\section{Harmonic measure estimates and derivative bounds for conformal maps}\label{harmestsec}

This section collects several estimates from~\cite{Ko18} regarding the relation of harmonic  measure and the derivatives of conformal maps. Some of them require an adjustment of parameters, in which case a proof is given; others are restated here for completeness, with a reference to a proof in~\cite{Ko18}.  

Given a domain $\Omega\subset \overline{\mathbb C}$, a point $\zeta\in \Omega$, and a Borel set $E\subset \partial \Omega$, let $\omega(\zeta, E, \Omega)$ be the harmonic measure of $E$ with respect to $\zeta$. This measure is determined by the property that $u(\zeta) = \int_{\partial \Omega} u(w)\,d\omega(\zeta, \cdot, \Omega)$  for every continuous function $u$ on $\overline{\Omega}$ that is harmonic in $\Omega$. The basic properties of harmonic measure can be found in~\cite{Ranb}.

\begin{lemma}\label{BNcor}\cite[Corollary 2.2]{Ko18} Let $\Omega\subset \CC$ be a simply connected domain. Consider a point $\zeta \in \Omega$ and a subset $\Gamma\subset \partial \Omega$. Suppose that $\omega(\zeta,\Gamma,\Omega)\ge \epsilon>0$. Then 
\begin{equation}\label{BNcor1}
\dist(\zeta,\Gamma) \le \csc^2 \left(\frac{\pi \epsilon}{4}\right) \dist(\zeta,\partial\Omega);
\end{equation}
\begin{equation}\label{BNcor2}
\diam \Gamma \ge \tan^2 \left(\frac{\pi \epsilon}{4} \right) \dist(\zeta,\Gamma).
\end{equation}
\end{lemma}
 
In order to use Lemma~\ref{BNcor}, we need to estimate the harmonic measure of the arcs~\eqref{gammacurves} from below. 

\begin{lemma}\label{lowerharm} Using notation~\eqref{gammacurves}, we have
\begin{equation}\label{lowerharm1}
\omega(z, \gamma_j, \DD) \ge \frac{1}{30\pi} \quad \text{if $j=1,4$ and $e^{-\pi/6}\le |z|<1$}
\end{equation}
and
\begin{equation}\label{lowerharm2}
\omega(z, \gamma_j, \DD) \ge \frac{1}{112\pi} \quad \text{if $j=2,3$ and $e^{-4\pi}\le |z|<1$}.
\end{equation}
\end{lemma}
\begin{proof}   Let $j\in \{2, 3\}$ and $\delta=\log(1/r)$. For $\zeta\in \gamma_j$ the Poisson kernel $P_z(\zeta)$ can be estimated from below as follows:
\[
P_z(\zeta) = \frac{1}{2\pi} \frac{1-|z|^2}{|\zeta-z|^2}\ge 
\frac{1}{2\pi} \frac{1-r}{|1 - r +\delta|^2}. 
\]
The convexity of the function $\phi(r) = \log(1/r)$, which has a tangent line $\psi(r) = 1-r$ at $r=1$, implies
\begin{equation}\label{logbounds2}
1 < \frac{\log(1/r)}{1-r} < \frac{\phi(e^{-4\pi})}{\psi(e^{-4\pi})} = 
\frac{4\pi}{1-e^{-4\pi}} < 13  \quad \text{for }\ e^{-4\pi} < r< 1.
\end{equation}
Using the inequalities $1-r < \delta < 13(1-r)$ we obtain $|1-r+\delta|^2 < 2\delta\cdot 14(1-r)$. This yields
$P_z(\zeta) \ge 1/( 56 \pi \delta)$. 
Since the length of  $\gamma_j$ is $\delta/2$, its harmonic measure is at least $1/(112\pi)$. This completes the proof of~\eqref{lowerharm2}.  

The proof of~\eqref{lowerharm1}, which is similar, can be found in~\cite[Lemma 2.3]{Ko18}.
\end{proof}
 
Our next step is to translate the information about harmonic measure into distortion estimates for conformal maps. 

\begin{lemma}\label{intconf}
 Let $\Omega\subset \CC$ be a domain bounded by a Jordan curve $\Gamma$. Fix a conformal map $\Phi $  of $\DD$ onto $\Omega$ and let $\phi\colon \TT\to\Gamma$ be the boundary map induced by $\Phi$. Also let $z\in\DD$ and $r=|z|$.  Then 
\begin{equation}\label{confderest1}
|\Phi'(z)| \ge  
\frac{\dist(\phi(\gamma_1), \phi(\gamma_4))}{60000\log(1/r)},\quad \text{if } e^{-\pi/6} \le r< 1, 
\end{equation}
\begin{equation}\label{confderest2}
| \Phi'(z) | \le 
2\cdot 10^7\,\dfrac{\diam \phi(\gamma_2) \wedge \diam \phi(\gamma_3)}{\log(1/r)},\quad \text{if }  e^{-4\pi} \le r< 1.
\end{equation}
where $\gamma_j$ is as in~\eqref{gammacurves}. 
\end{lemma}
  
\begin{proof} Let $\zeta = \Phi(z)$, $\rho = \dist(\zeta, \partial\Omega)$, and $\Gamma_j=\phi(\gamma_j)$. Applying the Koebe $1/4$-theorem to the disk centered at $z$ with radius $1-r$, we obtain 
$(1-r) |\Phi'(z)| \le 4\rho$, which  by~\eqref{logbounds2} implies
\begin{equation}\label{Koebe}
    |\Phi'(z)| \le\frac{52 \rho}{\log(1/r)}. 
\end{equation}
By the conformal invariance of harmonic measure, 
\begin{equation}\label{Koebe-new}
\omega(\zeta, \Gamma_j, \Omega)
= \omega(z, \gamma_j, \DD) \ge \frac{1}{112\pi},\quad j=2, 3,
\end{equation}
where the last step is based on~\eqref{lowerharm2}. 
From ~\eqref{BNcor2} and~\eqref{Koebe-new}  it follows that
\[
\diam \Gamma_2 \wedge \diam \Gamma_3 \ge  \tan^2 \left(\frac{1}{448}\right)\rho. 
\]
Hence 
\[\begin{split}
|\Phi'(z)|  &\le \frac{52 \rho}{\log(1/r)} \le 
 \frac{52 (\diam \Gamma_2 \wedge \diam \Gamma_3)}{\tan^2 \left(\frac{1}{448}\right)\log(1/r)} \\
&\le 2\cdot 10^7\, \frac{ \diam \Gamma_2 \wedge \diam \Gamma_3}{\log(1/r)}
\end{split}
\]
which proves~\eqref{confderest2}.

The proof of~\eqref{confderest1}, which is similar, can be found in ~\cite[Lemma 3.1]{Ko18}. 
\end{proof}

\section{Proof of Theorem~\ref{main-theorem}}\label{sec-main-theorem-proof}

We may and do assume that $f$ is sense-preserving. The Jordan curve $f(\TT)$ divides the plane into two domains, one of which, denoted $\Omega$, is bounded.

\textbf{Extension to the bounded component.}
Let $\Phi$ be a conformal map of $\DD$ onto $\Omega$. By Carath\'eodory's theorem, $\Phi$ extends to a homeomorphism between $\overline{\DD}$ and $\overline{\Omega}$. Let $\phi\colon \TT \to \partial\Omega $ be the induced boundary map. Define $\psi\colon \TT\to\TT$ by $\psi = f^{-1}\circ \phi$. By Lemma~\ref{normalize} there exists a M\"obius map $\mu\colon \DD\to\DD$ such that $\psi\circ \mu$ is normalized. By replacing $\Phi$ with $\Phi\circ \mu$ (and consequently $\phi$ with $\phi\circ \mu$) we ensure that $\psi$ itself is normalized.  

Let $\Psi\colon\DD\to\DD$ be the extension of $\psi$ provided by Lemma~\ref{homeodist}. The composition 
\begin{equation}\label{extension-F}
 F = \Phi\circ \Psi^{-1}   
\end{equation}
extends $f =\phi\circ \psi^{-1} $ to the unit disk. The assumption~\eqref{BLass} on $f$ takes the form
\begin{equation}\label{BLass1}
\ell \le    \frac{|\phi(a)-\phi(b)|}{|\psi(a)-\psi(b)|} \le L,\quad a, b\in\TT, \ a\ne b,
\end{equation}
which has implications for the images of the arcs defined by~\eqref{gammacurves}:
\begin{equation}\label{BLass-diam}
\frac{\diam \phi(\gamma_j) }{\diam  \psi(\gamma_j) }\le L,
\end{equation}
\begin{equation}\label{BLass-dist}
\frac{\dist(\phi(\gamma_1), \phi(\gamma_4))}{\dist(\psi(\gamma_1), \psi(\gamma_4))} \ge \ell. 
\end{equation}
The derivatives of $F$ will be estimated using the chain rule, which says that for $z\in\DD\setminus\{0\}$,
\begin{equation}\label{chain-derivative}
\|DF(z)\| = |\Phi'(z)| \|D\Psi(z)^{-1}\|, \quad   \|DF(z)^{-1}\| = |\Phi'(z)|^{-1} \|D\Psi(z)\|.
\end{equation}
The Koebe distortion theorem for conformal maps of the disk~\cite[Theorem 2.5]{Durb} states that
\begin{equation}\label{Koebedist1}
\frac{1-|z|}{(1+|z|)^3} |\Phi'(0)|  \le |\Phi'(z)| \le  \frac{1+|z|}{(1-|z|)^3} |\Phi'(0)|, \quad z\in\DD.
\end{equation}

\textit{A bound for $\|DF(z)\|$ when $e^{-4\pi} < |z| < 1$.} Combine~\eqref{BAderest2}, ~\eqref{confderest2}, ~\eqref{BLass-diam}, and ~\eqref{chain-derivative} to obtain
\begin{equation}\label{DF-1}
\|DF(z)\| \le 2\cdot 10^{14} L, \quad e^{-4\pi} < |z| < 1.  
\end{equation}

\textit{A bound for $\|DF(z)\|$ when $0<|z|\le e^{-4\pi}$.} 
By the Schwarz lemma, $|\Phi'(0)|\le \diam \Omega\le 2L$.  By virtue of~\eqref{Koebedist1} this implies 
\[
|\Phi'(z)| \le \frac{1+e^{-4\pi}}{(1-e^{-4\pi})^3}\,2L \le 3L.
\]
This together with~\eqref{BAderest2} and~\eqref{chain-derivative} imply
\begin{equation}\label{DF-2}
\|DF(z)\| \le 3\cdot 10^7 L, \quad 0<|z|\le e^{-4\pi}.
\end{equation}

\textit{A bound for $\|DF(z)^{-1}\|$ when $e^{-\pi/6} < |z| < 1$.}
Combine~\eqref{BAderest1}, ~\eqref{confderest1}, ~\eqref{BLass-dist}, and ~\eqref{chain-derivative} to obtain
\begin{equation}\label{DFinv-1}
\|DF(z)^{-1}\| \le 6\cdot 10^{11} \ell^{-1},\quad e^{-\pi/6} < |z| < 1.
\end{equation}

\textit{A bound for $\|DF(z)^{-1}\|$ when $0<|z|\le e^{-\pi/6}$.}
This case is more involved because a lower bound for $|\Phi'(0)|$ is not as readily available as an upper bound. We begin by partitioning $\TT$ into six arcs of length $\pi/3$. Since their images under $\psi$ cover $\TT$, at least one of them, denoted $\sigma$, must satisfy $|\psi(\sigma)|\ge \pi/3$. Let $z_0 = e^{-\pi/6} e^{i\theta}$ where $e^{i\theta}$ is the midpoint of $\sigma$.

Consider the arcs $\gamma_j$ associated with $z_0$ by~\eqref{gammacurves}.  
They are all contained in the arc $\sigma'$, which shares the midpoint with $\sigma$ and is twice its size, $|\sigma'|=2\pi/3$.
We need to estimate $\dist(\psi(\gamma_1), \psi(\gamma_4))$ from below. To this end, note that one of the two components of $\TT\setminus (\psi(\gamma_1)\cup \psi(\gamma_4))$ is the interior of $\psi(\sigma)$, and $\diam \psi(\sigma)\ge 1$. The other component is $\TT\setminus \psi(\sigma')$, which also has diameter at least $1$ because $|\psi(\sigma')|\le 4\pi /3$ by virtue of $\psi$ being normalized. Thus,
\[
\dist(\psi(\gamma_1), \psi(\gamma_4)) \ge 1.
\]
Recalling~\eqref{confderest1} and~\eqref{BLass-dist}, we arrive at 
\begin{equation}\label{10000pi}
|\Phi'(z_0)|\ge \frac{\ell}{60000 \log(1/|z_0|)} = \frac{\ell}{10000\pi}.     
\end{equation}
By the distortion theorem~\eqref{Koebedist1}, inequality~\eqref{10000pi} implies a similar lower bound for all $z$ with $|z|\le e^{-\pi/6}$.  
\begin{equation}\label{lower-bound}
|\Phi'(z)|\ge \left(\frac{1-e^{-\pi/6}}{1+e^{-\pi/6}} \right)^4 \frac{\ell}{10000\pi} \ge 10^{-7}\ell.     
\end{equation}
Finally, from~\eqref{BAderest1}, ~\eqref{chain-derivative} and~\eqref{lower-bound} we obtain
\begin{equation}\label{DFinv-2}
\|DF(z)^{-1}\| \le 10^{14}\ell^{-1}, \quad 0<|z|\le e^{-\pi/6}.
\end{equation}

To summarize, so far we have a homeomorphism $F$ between $\overline{\DD}$ and $\overline{\Omega}$ whose derivative is bounded by $2\cdot 10^{14}$ in $\DD\setminus\{0\}$ by virtue of~\eqref{DF-1} and~\eqref{DF-2}. This implies the upper Lipschitz bound 
\[
|F(a)-F(b)|\le 2\cdot 10^{14} |a-b|,\quad a,b\in\overline{\DD},
\]
since $a$ and $b$ can be joined by a segment that lies in $\DD$.

The Lipschitz property of $F^{-1}$ is less obvious because its domain $\overline{\Omega}$ need not be convex. If two points $a,b\in  \overline{\Omega}$ are such that the segment $[a, b]$ is  contained in $\Omega$, then the Lipschitz bound 
\[
|F^{-1}(a)-F^{-1}(b)|\le 10^{14}\ell^{-1}|a-b|
\]
follows from the derivative bounds~\eqref{DFinv-1} and ~\eqref{DFinv-2}. Otherwise, let $c, d$ be the first and last points at which this line segment meets $\partial \Omega$; that is $[a, c] \subset \overline{\Omega} $, $[d, b]\in \overline{\Omega}$, and $c,d\in \partial\Omega$. Since $F^{-1} = f^{-1}$ on $\partial\Omega$,  we have 
$|F^{-1}(c)-F^{-1}(d)|\le \ell^{-1}|c-d|$. By the triangle inequality, 
\[ \begin{split}
|F^{-1}(a)-F^{-1}(b)| & \le 10^{14}\ell^{-1}|a-c| + 
\ell^{-1}|c-d| + 10^{14}\ell^{-1}|d-b| 
\\ &\le 10^{14}\ell^{-1}|a-b|.
\end{split}\]
This completes the proof of~\eqref{main-conclusion}. 

\textbf{Extension to the unbounded component.}  Theorem~1.1 of~\cite{Ko18} yields the existence of a homeomorphic extension $\widetilde{F}\colon \CC\to\CC$ that satisfies~\eqref{main-conclusion2}. Strictly speaking, it is stated in the context of bi-Lipschitz maps, i.e., with $\ell,L\in (0, \infty)$. However, nothing prevents one from setting $\ell=0$ or $L=\infty$ in~\cite{Ko18} as long as $f$ is still assumed to be an embedding: the estimates that involve a degenerate constant become vacuously true, and the estimates that do not involve it are not affected. 

An extension that satisfies both~\eqref{main-conclusion} and~\eqref{main-conclusion2} is constructed by using  $\widetilde{F}$ in $\CC\setminus\overline{\DD}$ and $F$, as defined by~\eqref{extension-F}, in $\DD$. The estimates ~\eqref{main-conclusion2} continue to hold because $F$ satisfies them in $\DD$ by virtue of~\eqref{main-conclusion}.
This completes the proof of Theorem~\ref{main-theorem}.

\section{Examples, remarks, and questions} \label{sec:examples}

Given that bi-Lipschitz maps of a circle can be extended to the enclosed disk with linear estimates for both constants~\eqref{main-conclusion}, it is natural to expect the same for an extension to the exterior of the disk. This turns out to be false.   

\begin{example}\label{keyhole} Let $\epsilon>0$ be a small number and let $\Gamma$ be the union of the following concentric circular arcs:  
\[
\begin{split}
\sigma_1 &= \{e^{it}\colon \epsilon\le t\le 2\pi-\epsilon\} \\
\sigma_2 &= \{2e^{it}\colon \epsilon\le t\le \pi-\epsilon\} \\
\sigma_3 &= \{2e^{it}\colon \pi+\epsilon\le t\le 2\pi-\epsilon\} \\
\sigma_4 &= \{3e^{it}\colon -\pi+\epsilon\le t\le \pi-\epsilon\} 
\end{split}
\]
joined by radial line segments as in Figure~\ref{fig:keyhole}. 

\begin{figure}[h]
    \centering
    \includegraphics[width=0.6\textwidth]{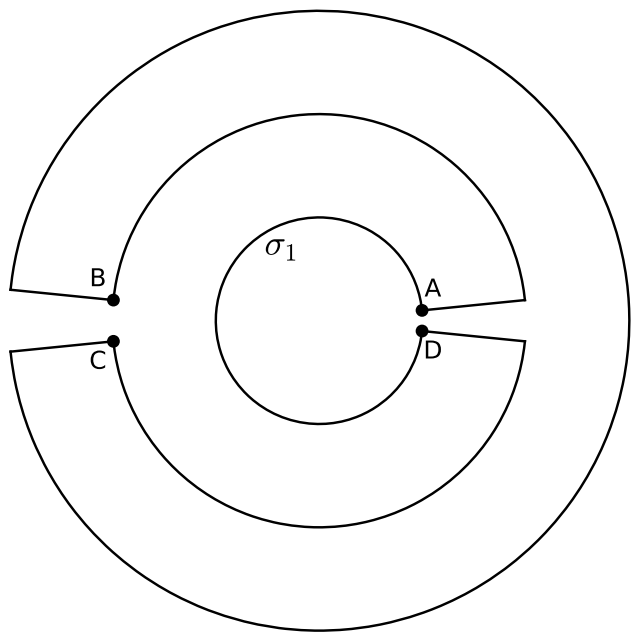}
    \caption{Impossibility of a linear lower Lipschitz bound}
    \label{fig:keyhole}
\end{figure}

Define $f\colon \TT\to\Gamma$ so that the restriction of $f$ to $\sigma_1$ is the identity map, and $f\colon \TT\setminus \sigma_1\to \Gamma\setminus\sigma_1$ is the constant-speed parameterization. Then $f$ satisfies the bi-Lipschitz condition~\eqref{BLass} with $L = C/\epsilon $ for some $C>0$ independent of $\epsilon$, and $\ell=1/2$. To justify the latter, note that $f$ does not contract any distances either on $\sigma_1$ or on $\TT\setminus\sigma_1$, and there is only moderate contraction between the points of $\sigma_1$ and their near-antipodes on $\TT$, from distance $\approx 2$ to distance $\ge 1$.  

Suppose that $F\colon \CC\to\CC$ is a homeomorphism extending $f$. Observe that the line segment $BC$ separates the arc $\sigma_1$ from $\infty$ in the unbounded component of $\CC\setminus \Gamma$. Therefore, its preimage $F^{-1}(BC)$ separates $\sigma_1$ from $\infty$ in the domain $\CC\setminus \overline{\DD}$. Since the endpoints of $F^{-1}(BC)$ lie on $\TT\setminus \sigma_1$, it follows that its length is at least $2\pi-2\epsilon$. On the other hand, the length of the segment $BC$ is less than $4\epsilon$. This shows that $\Lip(F^{-1})\ge (2\pi-2\epsilon)/(4\epsilon)$ which is unbounded as $\epsilon\to 0$, while $\Lip(f^{-1})\le 2$. 

Note that $L^2/\ell$ is of order  $1/\epsilon$ in this example. Thus, the left side of the inequality~\eqref{main-conclusion2} is sharp except for the numeric constant. 
\end{example}
 
Since Kirszbraun's theorem provides an extension preserving the Lipschitz constant, one may ask if the same can be achieved by homeomorphic extension of circle embeddings, at least within the enclosed disk. The answer is negative.  

\begin{example}\label{moon} Fix a number $a\in (0, 1/2)$ and define $f\colon \TT\to\CC$ by $f(x+iy) =\min(x, 2a-x) +iy$. This map reflects an arc of the unit circle about the vertical line $x=a$. Its image is the crescent shown in Figure~\ref{fig:moon}. 

The map $f$ is nonexpanding, i.e., $\Lip(f)=1$. Any homeomorphic extension $F\colon \overline{\DD}\to\overline{\DD}$ must send $0$ to a point $F(0)\ne 0$. Since $F(\pm i)=\pm i$, at least one of the distances $|F(i)-F(0)|$ and $|F(-i)-F(0)|$ exceeds $1$. Thus, $\Lip(F)>1$.

\begin{figure}[h]
    \centering
    \includegraphics[width=0.3\textwidth]{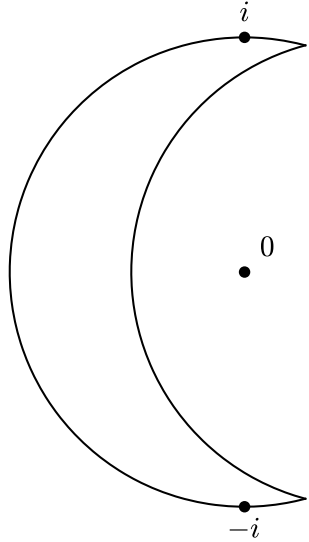}
    \caption{Impossibility of preserving the upper Lipschitz bound}
    \label{fig:moon}
\end{figure}
\end{example}
  
By letting $a\to 0$ in Example~\ref{moon} one can see that a bound of the form $\Lip(F)\le C\Lip(f)$ must have $C\ge \sqrt{2}$. The modest size of the factor $\sqrt{2}$ in this example is in stark contrast with the large factors in Theorem~\ref{main-theorem}. 

\begin{question} Can the inequality~\eqref{main-conclusion} be improved so that it has single-digit factors instead of 15-digit ones? For example, is there an extension that satisfies 
\begin{equation}\label{main-conclusion-better}
 \frac{\ell}{2} |a-b|\le  |F(a)-F(b)| \le 2  L |a-b|,\quad a,b \in\overline{\DD}?
\end{equation}
\end{question}

More reasonable estimates are available for the bi-Lipschitz constants of an extension of an embedding $f\colon\RR\to\CC$. 

\begin{theorem}\label{thm-line} ~\cite[Theorem 1.2]{Ko12} 
Suppose $f\colon\RR\to\CC$ is an embedding such that 
\begin{equation}\label{BLass-line}
\ell|a-b|\le  |f(a)-f(b)| \le L |a-b|,\quad a,b \in\RR 
\end{equation}
for some $\ell, L\in [0, \infty]$. Then $f$ extends to a homeomorphism $F\colon\CC\to\CC$ that  satisfies 
\begin{equation}\label{main-conclusion-line}
\frac{\ell}{120} |a-b|\le  |F(a)-F(b)| \le 2000 L |a-b|,\quad a,b \in\CC.
\end{equation}
\end{theorem}

It should be noted that \cite[Theorem 1.2]{Ko12} is stated in the context of bi-Lipschitz maps, i.e., $0<\ell,L<\infty$. But the argument presented in ~\cite{Ko12} works with $\ell=0$ or $L=\infty$ just as well, as long as $f$ is an embedding. 
  
Theorem~\ref{main-theorem} cannot be derived from Theorem~\ref{thm-line} by way of M\"obius conjugation, as such conjugation causes nonlinear growth of the bi-Lipschitz constants. For the same reason, neither theorem provides linear bounds for a bi-Lipschitz extension of an embedding $f\colon \TT\to S^2$ when the sphere $S^2$ is equpped with the spherical metric. 

\bigskip
\footnotesize
\noindent\textit{Acknowledgments.}
This research was supported by the National Science Foundation grants DMS-1362453 and DMS-1764266.

\bibliographystyle{amsplain}

\end{document}